\documentclass[a4paper, 12pt,oneside,reqno]{amsart}
\usepackage[a4paper]{geometry}
\usepackage{amssymb,amsmath,amsthm}
\usepackage[T2A]{fontenc}
\usepackage[breaklinks=true]{hyperref}

\usepackage{tikz}
\usetikzlibrary{decorations.markings,arrows.meta}

\usepackage{xcolor}

\usepackage{bbm}
\usepackage{graphicx}

\def\dim{\mathrm{dim}}

\def\As{\mathrm{As}}

\let\<\langle
\let\>\rangle

\newtheorem{definition}{Definition}
\newtheorem{lemma}{Lemma}

\newtheorem{theorem}{Theorem}
\newtheorem{corollary}{Corollary}
\newtheorem{remark}{Remark}
\newtheorem{example}{Example}

\title{Free Algebras in Mal'cev-Type Subvarieties of Associative Algebras}

\author{B. Sartayev$^{*}$}

\address{Narxoz University, Almaty, Kazakhstan and SDU University, Kaskelen, Kazakhstan}

\email{baurjai@gmail.com}

\author{A. Ydyrys}

\address{SDU University, Kaskelen, Kazakhstan}

\email{abdibek54@gmail.com}



\thanks{${}^{*}$Corresponding author: Bauyrzhan Sartayev   (baurjai@gmail.com)}

\begin{document}

\maketitle

\begin{abstract}
In this paper, we study free algebras in subvarieties of the variety of associative algebras singled out by Mal'cev’s classification. For each subvariety, we construct the bases for the corresponding free algebras and describe the space of symmetric polynomials they contain. 
\end{abstract}

\section{Introduction}

Associative algebras are closely related to several important classes, including Lie, Jordan, dendriform, and Novikov-type structures. The Poincaré–Birkhoff–Witt theorem embeds any Lie algebra into an associative algebra under the commutator, while many Jordan algebras arise from associative algebras via the anti-commutator; however, there exist exceptional Jordan algebras that are not realized in this way \cite{Cohn}. Beyond such embeddings, operads provide a uniform framework for passing between varieties by functorial constructions and for organizing generators, relations, and identities \cite{RB2,GK94,RB1,MarklRemm}. This operadic viewpoint has proved effective across a broad range of nonassociative settings; see, e.g., \cite{DauSar,Remm1}.

A natural source of subvarieties within the variety $\As$ of associative algebras is provided by additional polynomial identities. Mal’cev’s classification of degree-$3$ identities \cite{Mal'cev} isolates several canonical types, each generating a subvariety with its own combinatorics and representation-theoretic flavor. In fact, these identities are generated by invariant relations in associative algebras.
In \cite{KazMat}, the variety of alternative algebras was considered, following Malcev's classification. In \cite{4-type}, some properties of associative algebras with Malcev's classification are given. In this paper, we focus on two of these classes—the second and the third types—and study their free objects on the infinite countable set $X=\{x_1,\dots,x_n\}$. As proved in \cite{Lopatin}, the relatively free associative algebra is nilpotent; thus, so a separate study of free objects in the first-type variety is unnecessary. These types of associative algebras have a connection with classical algebras such as alternative, assosymmetric and $(-1,1)$ algebras \cite{assos1, Dzhuma, -1-1}. For recent results on these algebras, see \cite{perm2, Alt3, assos3}.
Indeed, this connection is explained by Koszul duality. If $\mathcal P$ is a binary quadratic operad and $\mathcal P^{!}$ its Koszul dual, then for any $\mathcal P$–algebra $U$ and any $\mathcal P^{!}$–algebra $V$ the bilinear operation
\[
[u\otimes x,\, v\otimes y]\;:=\;(uv)\otimes(xy)\;-\;(vu)\otimes(yx)
\]
(endowed with the given multiplications in $U$ and $V$) defines a Lie bracket on $U\otimes V$, see \cite{GK94}. In particular, since the dual operad of the first-type associative operad $\mathcal{A}s_1$ is the alternative operad, if $U$ is a first-type associative algebra and $V$ is alternative, then $U\otimes V$ becomes a Lie algebra under the bracket above.

We first construct the bases for the relatively free algebras in the second- and third-type Mal’cev subvarieties of the variety of associative algebras. Having a concrete basis (normal forms) is a standard tool in universal algebra and PI-theory: it is the starting point for solving the embedding problems, the Specht problem, the description of symmetric polynomials, the description problem of unary/binary subvarieties, and questions about tame versus wild automorphisms, etc. Let us list several related results in this direction:
\begin{itemize}
  \item Using a basis of the free perm algebra, it was shown in \cite{MS2022} that every metabelian Lie algebra can be embedded into an appropriate perm algebra. See also \cite{perm1,perm2} for mutations of perm algebras and the binary perm variety.
  \item Based on a basis of the free Zinbiel algebra, in \cite{KolMashSar}, it was proved that there exists a pre–Novikov algebra that does not embed into a Zinbiel algebra with a derivation. The space of Lie elements in the free Zinbiel algebra was studied in \cite{tortkara}, and mono/binary Zinbiel varieties were described in \cite{Zinb_2}.
  \item Bases for the free metabelian transposed Poisson and F-manifold algebras were constructed in \cite{AbdSart}. Automorphisms of finitely generated free metabelian Novikov algebras were studied in \cite{AutNov}; it is natural to pose the same problem for the metabelian transposed Poisson and F-manifold cases.
  \item It was proved in \cite{bicom1} that any metabelian Lie algebra can be embedded into an appropriate bicommutative algebra.
  \item Using a basis of the free Novikov algebra, in \cite{DotIsmUmi2023}, the Specht property for Novikov algebras was established. For the variety of noncommutative Novikov algebras case, in \cite{erlagol2021}, it was shown that every algebra in this variety embeds into an associative algebra with a derivation.
  \item For free metabelian Poisson algebras, symmetric polynomials were described via a basis approach in \cite{Findik1}.
  \item Using a Hall basis of the free Lie algebra, the Nowicki conjecture for the two-generated free Lie algebra was considered in \cite{Nowicki}.
\end{itemize}

The paper is organized as follows: 

Section~2 introduces the second-type identity and records the basic rewriting principles that will be used throughout. We construct a basis of the free second type-associative algebra and provide a combinatorial explanation of this basis. In Section 3, we describe the space of symmetric polynomials in the second-type associative operad. Section~4 addresses the third type in the commutator and anti-commutator bracket language. We construct the basis of the free third-type associative algebra. As an application, we describe the list of all independent identities under the commutator and anti-commutator.

In this paper, all algebras are defined over a field of characteristic $0$.

\section{Second-type associative algebra}

An associative algebra is called of the second type if it satisfies the following identity:
\begin{gather}\label{as2}
abc-acb-bac+bca+cab-cba=0.    
\end{gather}

By $\mathcal{A}s_2$, we denote the variety of second-type associative algebras. The corresponding free algebra of this variety is denoted by $\mathcal{A}s_2\<X\>$.

\begin{lemma}
In algebra $\mathcal{A}s_2\<X\>$ the following identities hold:
\begin{gather}\label{iddeg4}
dcab = dcba + cdab - cdba + adcb - acdb.
\end{gather}
and
\begin{gather}\label{id1}
abcde=abdce=acbde
\end{gather}
\end{lemma}
\begin{proof}
The result can be proven using computer algebra software programs, such as Albert \cite{Albert}.
\end{proof}

From (\ref{id1}) and (\ref{iddeg4}), we obtain
\begin{gather}\label{id2}
abcde=abced-bcaed+bcade.
\end{gather}
The identity (\ref{id1}) implies that in an associative monomial
\[
x_{i_1}x_{i_2}x_{i_3}\cdots x_{i_{n-1}}x_{i_n},
\]
the generators $x_{i_2},\;x_{i_3},\;\ldots,\;x_{i_{n-1}}$ can be ordered. So, the spanning set of the algebra $\mathcal{A}s_2\<X\>$ is the associative monomials with ordered generators given above. Every monomial from the spanning set is represented as follows:
\[
x_{i_1}x_{i_2}x_{i_3}\cdots x_{i_{n-1}}x_{i_n}\rightsquigarrow (x_{i_1},x_{i_n}).
\]
The identity (\ref{id2}) can be represented as
\begin{gather}\label{id2'}
(a,e)=(a,d)-(b,d)+(b,e).
\end{gather}
Since in (\ref{id2}), the generator \(c\) occurs in neither the first nor the last position, the rule (\ref{id2'}) applies to associative monomials of arbitrary length. Indeed, the rule (\ref{id2'}) can be illustrated as follows:

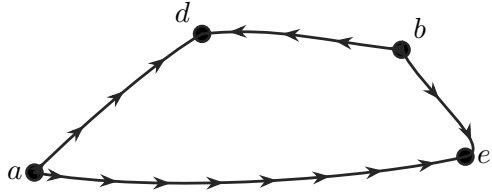
\begin{center}
\begin{tikzpicture}[scale=1.05,
  decoration={
    markings,
    mark=between positions 0.07 and 0.93 step 7mm
      with {\arrow{Stealth[length=2.3mm,width=1.7mm]}}
  },
  every path/.style={line cap=round, line join=round, draw=black!85, line width=1.05pt},
  every node/.style={font=\small}
]
  \coordinate (X) at (-3.2,-0.3);
  \coordinate (Z) at (-1.1, 1.45);
  \coordinate (T) at ( 1.4, 1.25);
  \coordinate (Y) at ( 2.2,-0.10);
  \node[left]        at (X) {$a$};
  \node[above left]  at (Z) {$d$};
  \node[above right] at (T) {$b$};
  \node[right]       at (Y) {$e$};
  \fill (X) circle (2.8pt);  
  \fill (Z) circle (2.8pt);  
  \fill (T) circle (2.8pt);  
  \fill (Y) circle (2.8pt);  
  \draw[postaction=decorate]
    (X) .. controls +(0.70,0.70) and +(-0.45,-0.25) .. (Z);
  \draw[postaction=decorate]
    (T) .. controls +(-0.95,0.10) and +(0.95,0.10) .. (Z);
  \draw[postaction=decorate]
    (T) .. controls +(0.25,-0.50) and +(0.40,0.10) .. (Y);
  \draw[postaction=decorate]
    (X) .. controls +(1.55,-0.28) and +(-1.20,-0.28) .. (Y);
\end{tikzpicture}
\end{center}
i.e., to reach the vertex $e$ from $a$, the rule proceeds through two intermediate arbitrary vertices (here $d$ and $b$). Since $(b,d)$ has the minus sign, the arrows are illustrated from $b$ to $d$. 

Let us define the sets $\mathcal{A}_1$, $\mathcal{A}_2$, $\mathcal{A}_3$ and $\mathcal{A}_4$ as follows:
$$\mathcal{A}_1=\{x_{i_1}\},\;\;\mathcal{A}_2=\{x_{j_1}x_{j_2},\;x_{j_2}x_{j_1}\},$$
$$\mathcal{A}_3=\{x_{k_2}x_{k_1}x_{k_3},\; x_{k_1}x_{k_3}x_{k_2},\; x_{k_3}x_{k_1}x_{k_2},\; x_{k_2}x_{k_3}x_{k_1},\; x_{k_3}x_{k_2}x_{k_1}\},$$
and
\begin{multline*}
\mathcal{A}_4=\{x_{l_1}x_{l_2}x_{l_3}x_{l_4}, x_{l_1}x_{l_2}x_{l_4}x_{l_3}, x_{l_1}x_{l_3}x_{l_4}x_{l_2}, x_{l_1}x_{l_4}x_{l_2}x_{l_3}, x_{l_1}x_{l_4}x_{l_3}x_{l_2},\\
x_{l_2}x_{l_3}x_{l_4}x_{l_1}, x_{l_3}x_{l_4}x_{l_1}x_{l_2}, x_{l_3}x_{l_4}x_{l_2}x_{l_1}, x_{l_4}x_{l_3}x_{l_2}x_{l_1}\},
\end{multline*}
where $k_1\leq k_2\leq k_3$ and $l_1\leq l_2\leq l_3\leq l_4$.
For $n\geq 5$, we define $\mathcal{A}_n$ as follows:
\[
\mathcal{A}_n=\{(x_{i_1},x_{r}),(x_{r},x_{i_1}),(x_{i_2},x_{i_3})\},
\]
where $i_1\leq r$, $i_1<i_2<i_3$ and $r=i_1,\;i_2,\;\ldots,\;i_n$. Also, we set
\[
\mathcal{A}=\bigcup_{i}\mathcal{A}_i.
\]
In the multilinear case,
\[
\mathcal{A}_5=\{(x_1,x_2),(x_1,x_3),(x_1,x_4),(x_1,x_5),(x_2,x_1),(x_3,x_1),(x_4,x_1),(x_5,x_1),(x_2,x_3)\}.
\]
and
\begin{multline*}
\mathcal{A}_6=\{(x_1,x_2),(x_1,x_3),(x_1,x_4),(x_1,x_5),(x_1,x_6),\\
(x_2,x_1),(x_3,x_1),(x_4,x_1),(x_5,x_1),(x_6,x_1),(x_2,x_3)\}.    
\end{multline*}

\begin{theorem}\label{basisAs2}
The set $\mathcal{A}$ is a basis of algebra $\mathcal{A}s_2\<X\>$.
\end{theorem}
\begin{proof}
Firstly, we show that the set $\mathcal{A}$ spans $\mathcal{A}s_2\<X\>$. For degrees up to $4$, the result can be verified by software programs, such as Albert \cite{Albert}. For monomials of degree $5$, we use the notion
\[
x_{i_1}x_{i_2}x_{i_3}\cdots x_{i_{n-1}}x_{i_n}\rightsquigarrow (x_{i_1},x_{i_n})
\]
as described above. In the multilinear case, we have $20$ different pairs, and we explicitly show the way of rewriting $11$ pairs as a sum $9$ from $\mathcal{A}_5$. By (\ref{id2'}),
\[
(x_{i_5},x_{i_3})=(x_{i_5},x_{i_1})-(x_{i_2},x_{i_1})+(x_{i_2},x_{i_3}),
\]
\[
(x_{i_4},x_{i_3})=(x_{i_4},x_{i_1})-(x_{i_2},x_{i_1})+(x_{i_2},x_{i_3}).
\]
\[
(x_{i_2},x_{i_5})=(x_{i_2},x_{i_3})-(x_{i_1},x_{i_3})+(x_{i_1},x_{i_5}),
\]
\[
(x_{i_2},x_{i_4})=(x_{i_2},x_{i_3})-(x_{i_1},x_{i_3})+(x_{i_1},x_{i_4}),
\]
\[
(x_{i_3},x_{i_5})=(x_{i_3},x_{i_1})-(x_{i_2},x_{i_1})+(x_{i_2},x_{i_5})=(x_{i_3},x_{i_1})-(x_{i_2},x_{i_1})+(x_{i_2},x_{i_3})-(x_{i_1},x_{i_3})+(x_{i_1},x_{i_5}),
\]
\[
(x_{i_3},x_{i_4})=(x_{i_3},x_{i_1})-(x_{i_2},x_{i_1})+(x_{i_2},x_{i_4})=(x_{i_3},x_{i_1})-(x_{i_2},x_{i_1})+(x_{i_2},x_{i_3})-(x_{i_1},x_{i_3})+(x_{i_1},x_{i_4}),
\]
\[
(x_{i_3},x_{i_2})=(x_{i_3},x_{i_1})-(x_{i_4},x_{i_1})+(x_{i_4},x_{i_2})=(x_{i_3},x_{i_1})-(x_{i_2},x_{i_1})+(x_{i_2},x_{i_3})-(x_{i_1},x_{i_3})+(x_{i_1},x_{i_2}),
\]
\[
(x_{i_4},x_{i_5})=(x_{i_4},x_{i_1})-(x_{i_2},x_{i_1})+(x_{i_2},x_{i_5})=(x_{i_4},x_{i_1})-(x_{i_2},x_{i_1})+(x_{i_2},x_{i_3})-(x_{i_1},x_{i_3})+(x_{i_1},x_{i_5}),
\]
\[
(x_{i_4},x_{i_2})=(x_{i_4},x_{i_5})-(x_{i_1},x_{i_5})+(x_{i_1},x_{i_2})=(x_{i_4},x_{i_1})-(x_{i_2},x_{i_1})+(x_{i_2},x_{i_3})-(x_{i_1},x_{i_3})+(x_{i_1},x_{i_2}),
\]
\[
(x_{i_5},x_{i_4})=(x_{i_5},x_{i_1})-(x_{i_2},x_{i_1})+(x_{i_2},x_{i_4})=(x_{i_5},x_{i_1})-(x_{i_2},x_{i_1})+(x_{i_2},x_{i_3})-(x_{i_1},x_{i_3})+(x_{i_1},x_{i_4})
\]
and
\[
(x_{i_5},x_{i_2})=(x_{i_5},x_{i_1})-(x_{i_4},x_{i_1})+(x_{i_4},x_{i_2})=(x_{i_5},x_{i_1})-(x_{i_2},x_{i_1})+(x_{i_2},x_{i_3})-(x_{i_1},x_{i_3})+(x_{i_1},x_{i_2}).
\]
For non-multilinear cases, the result works analogously. Indeed, these rewriting rules work for monomials of arbitrary degree. In the consequent degrees, with a new generator $x_{i_p}$, we add two additional monomials $(x_{i_1},x_{i_p})$ and $(x_{i_p},x_{i_1})$ to the basis.

So, it remains to show that any monomial of the form $(x_{i_t},x_{i_p})$ can be written as a sum of $\mathcal{A}$, where $t,\;p\neq 1$. If $t=x_{i_{2}}$ or $p=x_{i_{3}}$, then
\[
(x_{i_2},x_{i_p})=(x_{i_2},x_{i_3})-(x_{i_1},x_{i_3})+(x_{i_1},x_{i_p})
\]
and
\[
(x_{i_t},x_{i_3})=(x_{i_t},x_{i_1})-(x_{i_2},x_{i_1})+(x_{i_2},x_{i_3}).
\]
Otherwise,
\[
(x_{i_t},x_{i_p})=(x_{i_t},x_{i_1})-(x_{i_2},x_{i_1})+(x_{i_2},x_{i_3})-(x_{i_1},x_{i_3})+(x_{i_1},x_{i_p}).
\]

To prove the linear independence of the set $\mathcal{A}$, we construct the multiplication table for the algebra $A\<X\>$ with the basis $\mathcal{A}$ and show that such an algebra $A\<X\>$ belongs to the variety of second-type associative algebras. Up to degree 4, we define multiplication in $A\<X\>$ that is consistent with the associative identity and identity (\ref{as2}). Starting from degree $5$, the multiplication is defined as follows:
\[
\begin{cases}
(x_{i_1},x_{i_b})(x_{i_c},x_{i_d})=(x_{i_1},x_{i_d})\\
(x_{i_a},x_{i_1})(x_{i_1},x_{i_d})=(x_{i_a},x_{i_1})-(x_{i_2},x_{i_1})+(x_{i_2},x_{i_3})-(x_{i_1},x_{i_3})+(x_{i_1},x_{i_d}),\;a,d\neq 1\\
(x_{i_a},x_{i_b})(x_{i_c},x_{i_1})=(x_{i_a},x_{i_1})\\
(x_{i_a},x_{i_1})(x_{i_2},x_{i_3})=(x_{i_a},x_{i_1})-(x_{i_2},x_{i_1})+(x_{i_2},x_{i_3})-(x_{i_1},x_{i_3})+(x_{i_1},x_{i_3}),\;a\neq 1,2\\
(x_{i_2},x_{i_3})(x_{i_2},x_{i_3})=(x_{i_2},x_{i_3})
\end{cases}
\]
The multiplication of the basis monomial with a generator is defined analogously.

It remains to verify that $A\<X\>$ satisfies the defining identities of the algebra $\mathcal{A}s_2\<X\>$. We set 
\[
a=(x_{i_p},x_{i_t}),\;b=(x_{i_s},x_{i_r})\; \textrm{and}\; a=(x_{i_m},x_{i_n}).
\]
To check the associative identity, we consider $4$ cases:

Case $p=1$: 
\begin{multline*}
((x_{i_1},x_{i_t})(x_{i_s},x_{i_r}))(x_{i_m},x_{i_n})-(x_{i_1},x_{i_t})((x_{i_s},x_{i_r})(x_{i_m},x_{i_n}))=\\(x_{i_1},x_{i_r})(x_{i_m},x_{i_n})-(x_{i_1},x_{i_t})((x_{i_s},x_{i_1})-(x_{i_2},x_{i_1})+(x_{i_2},x_{i_3})-(x_{i_1},x_{i_3})+(x_{i_1},x_{i_n}))=\\
(x_{i_1},x_{i_n})-(x_{i_1},x_{i_1})+(x_{i_1},x_{i_1})-(x_{i_1},x_{i_3})+(x_{i_1},x_{i_3})-(x_{i_1},x_{i_n})=0.
\end{multline*}

Case $n=1$: 
\begin{multline*}
((x_{i_p},x_{i_t})(x_{i_s},x_{i_r}))(x_{i_m},x_{i_1})-(x_{i_p},x_{i_t})((x_{i_s},x_{i_r})(x_{i_m},x_{i_1}))=\\
((x_{i_p},x_{i_1})-(x_{i_2},x_{i_1})+(x_{i_2},x_{i_3})-(x_{i_1},x_{i_3})+(x_{i_1},x_{i_r}))(x_{i_m},x_{i_1})-(x_{i_p},x_{i_t})(x_{i_s},x_{i_1})=\\
(x_{i_p},x_{i_1})-(x_{i_2},x_{i_1})+(x_{i_2},x_{i_1})-(x_{i_1},x_{i_1})+(x_{i_1},x_{i_1})-(x_{i_p},x_{i_1})=0.
\end{multline*}

Case $p,n,r\neq 1$ and $s=1$:
\begin{multline*}
((x_{i_p},x_{i_1})(x_{i_1},x_{i_r}))(x_{i_1},x_{i_n})=(x_{i_p},x_{i_1})(x_{i_1},x_{i_n})-(x_{i_2},x_{i_1})(x_{i_1},x_{i_n})\\
+(x_{i_2},x_{i_3})(x_{i_1},x_{i_n})-(x_{i_1},x_{i_3})(x_{i_1},x_{i_n})+(x_{i_1},x_{i_r})(x_{i_1},x_{i_n})=\\
(x_{i_p},x_{i_1})-(x_{i_2},x_{i_1})+(x_{i_2},x_{i_3})-(x_{i_1},x_{i_3})+(x_{i_1},x_{i_n})\\
-((x_{i_2},x_{i_1})-(x_{i_2},x_{i_1})+(x_{i_2},x_{i_3})-(x_{i_1},x_{i_3})+(x_{i_1},x_{i_n}))\\
+(x_{i_2},x_{i_1})-(x_{i_2},x_{i_1})+(x_{i_2},x_{i_3})-(x_{i_1},x_{i_3})+(x_{i_1},x_{i_n})-(x_{i_1},x_{i_n})+(x_{i_1},x_{i_n})=\\
(x_{i_p},x_{i_1})-(x_{i_2},x_{i_1})+(x_{i_2},x_{i_3})-(x_{i_1},x_{i_3})+(x_{i_1},x_{i_n})
\end{multline*}
and
\begin{multline*}
(x_{i_p},x_{i_1})((x_{i_1},x_{i_r})(x_{i_1},x_{i_n}))=(x_{i_p},x_{i_1})(x_{i_1},x_{i_n})=\\
(x_{i_p},x_{i_1})-(x_{i_2},x_{i_1})+(x_{i_2},x_{i_3})-(x_{i_1},x_{i_3})+(x_{i_1},x_{i_n}).
\end{multline*}
On both sides, the result is the same.

Case $p,n,s\neq 1$ and $r=1$:
\begin{multline*}
((x_{i_p},x_{i_t})(x_{i_s},x_{i_1}))(x_{i_m},x_{i_n})=(x_{i_p},x_{i_1})(x_{i_m},x_{i_n})=\\
(x_{i_p},x_{i_1})-(x_{i_2},x_{i_1})+(x_{i_2},x_{i_3})-(x_{i_1},x_{i_3})+(x_{i_1},x_{i_n})
\end{multline*}
and
\begin{multline*}
(x_{i_p},x_{i_t})((x_{i_s},x_{i_1})(x_{i_m},x_{i_n}))=(x_{i_p},x_{i_t})(x_{i_s},x_{i_1})-(x_{i_p},x_{i_t})(x_{i_2},x_{i_1})\\
+(x_{i_p},x_{i_t})(x_{i_2},x_{i_3})-(x_{i_p},x_{i_t})(x_{i_1},x_{i_3})+(x_{i_p},x_{i_t})(x_{i_1},x_{i_n})=\\
(x_{i_p},x_{i_1})-(x_{i_p},x_{i_1})+(x_{i_p},x_{i_t})(x_{i_1},x_{i_n})=(x_{i_p},x_{i_1})-(x_{i_2},x_{i_1})+(x_{i_2},x_{i_3})-(x_{i_1},x_{i_3})+(x_{i_1},x_{i_n}).
\end{multline*}
On both sides, the result is the same. 

Case $p,s,m=2$ and $t,r,n=3$: 
\begin{multline*}
((x_{i_2},x_{i_3})(x_{i_2},x_{i_3}))(x_{i_2},x_{i_3})-(x_{i_2},x_{i_3})((x_{i_2},x_{i_3})(x_{i_2},x_{i_3}))=\\
(x_{i_2},x_{i_3})(x_{i_2},x_{i_3})-(x_{i_2},x_{i_3})(x_{i_2},x_{i_3})=(x_{i_2},x_{i_3})-(x_{i_2},x_{i_3})=0.
\end{multline*}
The identity (\ref{as2}) can be proved analogously. 
It remains to note that $A\<X\>\cong \mathcal{A}s_2\<X\>$.
\end{proof}

From Theorem \ref{basisAs2}, we have
\[
\dim(\mathcal{A}s_2(1))=1,\dim(\mathcal{A}s_2(2))=2,\dim(\mathcal{A}s_2(3))=5,\dim(\mathcal{A}s_2(4))=9
\]
and starting from degree $5$,
\[
\dim(\mathcal{A}s_2(n))=2n-1.
\]

\textbf{Combinatorial explanation of Theorem \ref{basisAs2}.}

Beginning with degree~$5$, regard the generator set $X$ as the vertex set of a directed graph, and interpret each monomial $(x_i,x_j)$ as a directed edge $x_i\to x_j$. In degree~$5$, the digraph $(X,\mathcal{A}_5)$ is strongly connected by (\ref{id2'}): from any $x_i$, there is a path to any $x_j$.

For degree~$6$, we adjoin a new vertex $x_6$. Since $x_6$ is initially isolated, to preserve strong connectivity, it is necessary and sufficient to add the two edges $(x_1,x_6)$ and $(x_6,x_1)$. Indeed, $x_1$ already reaches and is reachable from every other vertex via edges in $\mathcal{A}_5$, so for all $i,j$ we have paths
\[
x_i \rightsquigarrow x_1 \rightsquigarrow x_6
\quad\text{and}\quad
x_6 \rightsquigarrow x_1 \rightsquigarrow x_j.
\]
The same inductive step applies when adjoining each subsequent new vertex.

\section{Symmetric polynomials of operad $\mathcal{A}s_2$}

Let $X=\{x_1,x_2,\ldots,x_n\}$ be a finite set and let $p(x_1,\ldots,x_n)$ be a polynomial in the algebra $\mathcal{A}s_2\langle X\rangle$. We call $p$ symmetric if it is invariant under the action of the symmetric group $S_n$ on variables, i.e.,
\[
(\sigma\cdot p)(x_1,\ldots,x_n):=p(x_{\sigma(1)},\ldots,x_{\sigma(n)})=p(x_1,\ldots,x_n)\quad\text{for all }\sigma\in S_n.
\]
Define a family of polynomials $\mathcal{P}$ in $\mathcal{A}s_2\langle x_1,x_2,\ldots,x_n\rangle$ by
\[
p_1=\sum_i x_i,\qquad
p_2=\sum_{i\neq j} x_i x_j,
\]
\[
p_{3}=\sum_{i_1<i_2<i_3}\!\bigl(x_{i_1}x_{i_3}x_{i_2}+x_{i_3}x_{i_1}x_{i_2}+x_{i_2}x_{i_1}x_{i_3}\bigr),
\]
\begin{multline*}
p_4=\sum_{i_1<i_2<i_3<i_4}\bigl(
4\,x_{i_1}x_{i_2}x_{i_3}x_{i_4}
-10\,x_{i_1}x_{i_2}x_{i_4}x_{i_3}
+15\,x_{i_1}x_{i_4}x_{i_2}x_{i_3}
-5\,x_{i_1}x_{i_4}x_{i_3}x_{i_2}\\
\quad+3\,x_{i_2}x_{i_3}x_{i_4}x_{i_1}
+11\,x_{i_3}x_{i_4}x_{i_1}x_{i_2}
-8\,x_{i_3}x_{i_4}x_{i_2}x_{i_1}
+5\,x_{i_4}x_{i_3}x_{i_2}x_{i_1}\bigr),
\end{multline*}
and, for $n\ge 5$,
\[
p_n=\sum_{t\neq 2}(x_{i_t},x_{i_1})
+\sum_{t\neq 3}(x_{i_1},x_{i_t})
-(n-3)(x_{i_1},x_{i_3})
-(n-3)(x_{i_2},x_{i_1})
+(n-2)(x_{i_2},x_{i_3}).
\]

By the second-type associative operad we mean the multilinear component of $\mathcal{A}s_2\langle X\rangle$, i.e., the span of elements in which each $x_i$ occurs exactly once.

\begin{example}
In $\mathcal{A}s_2$, we obtain
\[
p_{3}=x_{1}x_{3}x_{2}+x_{3}x_{1}x_{2}+x_{2}x_{1}x_{3},
\]
\begin{multline*}
p_4=4x_{1}x_{2}x_{3}x_{4}-10x_{1}x_{2}x_{4}x_{3}+15x_{1}x_{4}x_{2}x_{3}-5x_{1}x_{4}x_{3}x_{2}
+3x_{2}x_{3}x_{4}x_{1}\\
+11x_{3}x_{4}x_{1}x_{2}-8x_{3}x_{4}x_{2}x_{1}+5x_{4}x_{3}x_{2}x_{1},
\end{multline*}
\[
p_5=(x_{3},x_{1})+(x_{4},x_{1})+(x_{5},x_{1})+(x_{1},x_{2})+(x_{1},x_{4})+(x_{1},x_{5})
-2(x_{1},x_{3})-2(x_{2},x_{1})+3(x_{2},x_{3}),
\]
and
\begin{multline*}
p_6=(x_{3},x_{1})+(x_{4},x_{1})+(x_{5},x_{1})+(x_{6},x_{1})+(x_{1},x_{2})+(x_{1},x_{4})+(x_{1},x_{5})+(x_{1},x_{6})\\
-3(x_{1},x_{3})-3(x_{2},x_{1})+4(x_{2},x_{3}).
\end{multline*}
\end{example}

\begin{theorem}
In the second-type associative operad, the symmetric polynomials are precisely the elements of the family $\mathcal{P}$.
\end{theorem}

\begin{proof}
The statement for degrees $\le 4$ is immediate. For $n\ge 5$ we must show that $p_n$ is $S_n$–invariant. Consider first the full symmetrization in the associative operad,
\[
\sum_{\sigma\in S_n} \sigma(x_1x_2\cdots x_n).
\]
By~(\ref{id1}),
\[
\sum_{\sigma\in S_n} \sigma(x_1x_2\cdots x_n)=(n-2)!\sum_{i\neq j}(x_i,x_j).
\]
For $n=5$, applying the rewriting rules from Theorem~\ref{basisAs2} yields
\[
6\sum_{i\neq j}(x_i,x_j)=4\,p_5,
\]
which proves the claim for $n=5$. In the general case, for arbitrary $n$, we use the general rewriting rules from Theorem \ref{basisAs2}, i.e,
\[
(x_{2},x_{p})=(x_{2},x_{3})-(x_{1},x_{3})+(x_{1},x_{p}),
\]
\[
(x_{t},x_{3})=(x_{t},x_{1})-(x_{2},x_{1})+(x_{2},x_{3}),
\]
and
\[
(x_{t},x_{p})=(x_{t},x_{1})-(x_{2},x_{1})+(x_{2},x_{3})-(x_{1},x_{3})+(x_{1},x_{p}).
\]
Substituting these into $(n-2)!\sum_{i\neq j}(x_i,x_j)$ and collecting terms gives
\[
(n-2)!\sum_{i\neq j}(x_i,x_j)=(n-1)\,p_n,
\]
hence, $p_n$ is symmetric. This completes the proof.
\end{proof}

\section{Third-type associative algebra}

\begin{definition}
An associative algebra is called of the third type if it satisfies the following identity:
\begin{gather}\label{as3}
abc+bac-bca-cba=0.    
\end{gather}
\end{definition}

By $\mathcal{A}s_3$, we denote the variety of third-type associative algebras. The corresponding free algebra of this variety is denoted by $\mathcal{A}s_3\<X\>$.

Let us rewrite every monomial of the algebra $\mathcal{A}s_3\<X\>$ in terms of the commutators and anti-commutators, i.e.,
\[
ab=1/2([a,b]+\{a,b\}).
\]
For such monomials, we have the following lemmas:
\begin{lemma}\label{id1as3}
\cite{4-type} An algebra $\mathcal{A}s_2^{(+)}\<X\>$ satisfies the following identities:
\[
\{\{\{a,b\},c\},d\}=\{\{\{a,c\},b\},d\}=\{\{\{a,b\},d\},c\}=\{\{a,b\},\{c,d\}\}.
\]
\end{lemma}
\begin{lemma}\label{id2as3}
In algebra $\mathcal{A}s_3\<X\>$ the following identities hold:
\[
[[a,b],[c,d]]=[[a,b],\{c,d\}]=[\{a,b\},\{c,d\}]=
\{\{a,b\},[c,d]\}=\{[a,b],[c,d]\}=0
\]
and
\begin{multline*}
[[[a,b],c],d]=\{[[a,b],c],d\}=[\{[a,b],c\},d]=[[\{a,b\},c],d]=\\
[\{\{a,b\},c\},d]=\{[\{a,b\},c],d\}=\{\{[a,b],c\},d\}=0.
\end{multline*}
\end{lemma}
\begin{proof}

The proof of 
\[
[[a,b],[c,d]]=[[[a,b],c],d]=0
\]
was given in \cite{4-type}. Also, in \cite{4-type}, it was proved that the polarization of $\mathcal{A}s_3$ which is
\begin{gather}\label{pol1}
[\{a,c\},b]=\{[a,b],c\}-\{a,[b,c]\},    
\end{gather}
\begin{gather}\label{pol2}
[b,[a,c]]=\{\{a,b\},c\}-\{a,\{b,c\}\},
\end{gather}
\begin{gather}\label{pol3}
\{a,\{b,c\}\}-3/2\{[a,c],b\}-3/2\{[a,b],c\}-1/2\{\{a,c\},b\}-1/2\{\{a,b\},c\}=0
\end{gather}
and
\begin{gather}\label{pol4}
\{a,[b,c]\}=-1/2\{[a,c],b\}+1/2\{[a,b],c\}+ 1/2\{\{a,c\},b\}-1/2\{\{a,b\},c\}.
\end{gather}
Firstly, we prove the identities that contain only one anti-commutator bracket.
By (\ref{pol1}), (\ref{pol2}), (\ref{pol3}), (\ref{pol4}) and Lemma \ref{id1as3}, we obtain
\[
\{b,[a,[c,d]]\}=^{(\ref{pol2})}\{b,\{\{c,a\},d\}\}-\{b,\{c,\{a,d\}\}\}=0\Rightarrow \{b,[a,[c,d]]\}=0,
\]
\[
[b,[a,\{c,d\}]]=^{(\ref{pol2})}\{\{a,b\},\{c,d\}\}-\{a,\{b,\{c,d\}\}\}=0\Rightarrow [b,[a,\{c,d\}]]=0,
\]
\[
[\{b,d\},[a,c]]=^{(\ref{pol2})} \{\{a,\{b,d\}\},c\}-\{a,\{\{b,d\},c\}\}=0\Rightarrow [\{b,d\},[a,c]]=0,
\]
\begin{multline*}
\{\{a,d\},\{b,c\}\}=^{(\ref{pol3})} 3/2\{[\{a,d\},c],b\}+3/2\{[\{a,d\},b],c\}+\{\{\{a,d\},b\},c\}\Rightarrow\\
\{[\{a,d\},c],b\}=-\{[\{a,d\},b],c\},    
\end{multline*}
\[
0=[[b,d],[a,c]]=^{(\ref{pol2})} \{\{[b,d],a\},c\}-\{\{[b,d],c\},a\}\Rightarrow\\
\{\{[b,d],a\},c\}=\{\{[b,d],c\},a\},  
\]
\begin{multline*}
\{[a,d],[b,c]\}=^{(\ref{pol4})}
-1/2\{[[a,d],c],b\}+1/2\{[[a,d],b],c\}+ 1/2\{\{[a,d],c\},b\}-1/2\{\{[a,d],b\},c\}=\\
1/2\{\{[a,d],c\},b\}-1/2\{\{[a,d],b\},c\}=0 \Rightarrow
\{[a,d],[b,c]\}=0
\end{multline*}
and
\[
[\{a,[c,d]\},b]=^{(\ref{pol1})} \{[a,b],[c,d]\}-\{a,[b,[c,d]]\}=0 \Rightarrow
[b,\{a,[c,d]\}]=0.
\]
Now, we prove the identities that contain exactly two anti-commutator brackets.

\[
0=[b,[a,[c,d]]]=^{(\ref{pol2})} \{\{a,b\},[c,d]\}-\{a,\{b,[c,d]\}\}\Rightarrow \{\{a,b\},[c,d]\}=\{a,\{b,[c,d]\}\}
\]
\[
0=[[a,b],[c,d]]=^{(\ref{pol2})}
\{\{[a,b],c\},d\}-\{\{[a,b],d\},c\}\Rightarrow \{\{[a,b],c\},d\}=\{\{[a,b],d\},c\}
\]
Applying (\ref{pol3}), to
\[
\{a,\{b,\{c,d\}\}\},\; \{\{a,d\},\{b,c\}\}\;\textrm{and}\;\{d,\{a,\{b,c\}\}\},
\]
we obtain
\[
\{[a,\{c,d\}],b\}=-\{[a,b],\{c,d\}\}
\]
\[
\{[\{a,d\},c],b\}=-\{[\{a,d\},b],c\}
\]
and
\[
\{d,\{[a,c],b\}\}=-\{d,\{[a,b],c\}\}.
\]
Also, we have
\[
[\{\{a,b\},c\},d]=^{(\ref{pol1})}\{[\{a,b\},d],c\}-\{\{a,b\},[d,c]\}\Rightarrow [\{\{a,b\},c\},d]=2\{[\{a,b\},d],c\},
\]
\[
[\{\{a,b\},c\},d]=[\{\{a,b\},c\},d]=2\{[\{a,b\},d],c\}=2\{[\{a,c\},d],b\},
\]
\[
[\{a,c\},\{b,d\}]=^{(\ref{pol1})} \{[a,\{b,d\}],c\}-\{a,[\{b,d\},c]\}\Rightarrow 2[\{a,c\},\{b,d\}]=\{[a,\{b,d\}],c\}.
\]
Using all obtained identities, from the one side, we have
\begin{multline*}
[a,\{b,\{c,d\}\}]=^{(\ref{pol1})}
\{[a,\{c,d\}],b\}+\{[a,b],\{c,d\}\}=^{(\ref{pol1}),(\ref{pol3})} \{\{[a,d],c\},b\}+\{\{[a,c],d\},b\}\\
+3/2\{[[a,b],d],c\}+3/2\{[[a,b],c],d\}+1/2\{\{[a,b],d\},c\}+1/2\{\{[a,b],c\},d\}=\{\{[a,b],c\},d\}.
\end{multline*}
From the other side,
\[
[a,\{b,\{c,d\}\}]=2\{b,[a,\{c,d\}]\}=2\{\{[a,b],c\},d\}.
\]
Equalizing both sides, we obtain
\[
\{\{[a,b],c\},d\}=0,
\]
and before obtained equalities complete the proof.
\end{proof}

Let us define the sets $\mathcal{B}_1$, $\mathcal{B}_2$, $\mathcal{B}_3$, and for $n\geq 4$, $\mathcal{B}_n$ as follows:
$$\mathcal{B}_1=\{x_{i_1}\},\;\;\mathcal{B}_2=\Bigl\{[x_{j_1},x_{j_2}],\;\{x_{j_1},x_{j_2}\}\Bigl\},$$
$$\mathcal{B}_3=\Bigl\{\{[x_{k_1},x_{k_2}],x_{k_3}\},\;\{[x_{k_1},x_{k_3}],x_{k_2}\},\;\{\{x_{k_1},x_{k_2}\},x_{k_3}\},\;\{\{x_{k_1},x_{k_3}\},x_{k_2}\}
\Bigl\}$$
and
\[
\mathcal{B}_n=\Bigl\{\{\{\cdots\{\{x_{l_1},x_{l_2}\},x_{l_3}\},\cdots,\},x_{l_n}\}\Bigl\},
\]
where $j_1\leq j_2$, $k_1\leq k_2\leq k_3$ and $l_1\leq l_2\leq l_3\leq\ldots\leq l_n$.
Also, we set
\[
\mathcal{B}=\bigcup_{i}\mathcal{B}_i.
\]
\begin{theorem}\label{basisAs3}
The set $\mathcal{B}$ is a basis of algebra $\mathcal{A}s_3\<X\>$.
\end{theorem}
\begin{proof}
For degrees up to $3$, the result can be verified by software programs, such as Albert \cite{Albert}. Starting from degree $4$, the identities from Lemma \ref{id1as3} and Lemma \ref{id2as3} provide that the set $\mathcal{B}$ spans the algebra $\mathcal{A}s_3\<X\>$.

Let us consider an algebra $B\<X\>$ with a basis $\mathcal{B}$. Up to degree $3$, we define multiplication in $B\<X\>$ that is consistent with the identities (\ref{pol1}), (\ref{pol2}), (\ref{pol3}) and (\ref{pol4}). From degree $4$, it is 
\[
\begin{cases}
[\{\{\cdots\{\{x_{l_1},x_{l_2}\},x_{l_3}\},\cdots\},x_{l_n}\},\{\{\cdots\{\{x_{s_1},x_{s_2}\},x_{s_3}\},\cdots\},x_{s_m}\}]=0 \\
\{\{\{\cdots\{x_{l_1},x_{l_2}\},\cdots\},x_{l_n}\},\{\{\cdots\{x_{s_1},x_{s_2}\},\cdots\},x_{s_m}\}\}=\{\{\cdots\{x_{r_1},x_{r_2}\},\cdots\},x_{r_{n+m}}\}\\
\end{cases}
\]
where $\{r_1,r_2,\ldots,r_{n+m}\}=\{l_1,l_2,\ldots,l_{n},s_1,s_2,\ldots,s_{m}\}$ and $r_1\leq r_2\leq\ldots\leq r_{n+m}$. Starting from degree $4$, the remaining cases are also equal to $0$. 

It is remained to prove that an algebra $B\<X\>$ satisfies the identities (\ref{pol1}), (\ref{pol2}), (\ref{pol3}) and (\ref{pol4}). For the monomials up to degree $3$, the result follows from the construction.

Now, we check the identities for monomials $A_1,A_2,A_3\in \mathcal{B}$ with degrees greater than $3$:
\begin{itemize}
    \item The identity (\ref{pol1}) is always trivial in $B\<X\>$.
    \item From (\ref{pol2}), we obtain
    $$[A_2,[A_1,A_3]]-\{\{A_1,A_2\},A_3\}+\{A_1,\{A_2,A_3\}\}=-\{\{A_1,A_2\},A_3\}+\{\{A_1,A_2\},A_3\}=0.$$
    \item The identity (\ref{pol3}) gives
    \begin{multline*}
    \{A_1,\{A_2,A_3\}\}-3/2\{[A_1,A_3],A_2\}-3/2\{[A_1,A_2],A_3\}-1/2\{\{A_1,A_3\},A_2\}-1/2\{\{A_1,A_2\},A_3\}=\\
    \{\{A_1,A_2\},A_3\}-1/2\{\{A_1,A_2\},A_3\}-1/2\{\{A_1,A_2\},A_3\}=0.
    \end{multline*}
    \item The identity (\ref{pol4}) gives
    \begin{multline*}
    \{A_1,[A_2,A_3]\}+1/2\{[A_1,A_3],A_2\}-1/2\{[A_1,A_2],A_3\}-1/2\{\{A_1,A_3\},A_2\}+1/2\{\{A_1,A_2\},A_3\}=\\
    -1/2\{\{A_1,A_3\},A_2\}+1/2\{\{A_1,A_2\},A_3\}=-1/2\{\{A_1,A_2\},A_3\}+1/2\{\{A_1,A_2\},A_3\}=0.
    \end{multline*}
\end{itemize}
It remains to note that $B\<X\>\cong \mathcal{A}s_3\<X\>$.
\end{proof}

From Theorem \ref{basisAs3}, we have
\[
\dim(\mathcal{A}s_2(1))=1,\dim(\mathcal{A}s_2(2))=2,\dim(\mathcal{A}s_2(3))=4
\]
and starting from degree $4$,
\[
\dim(\mathcal{A}s_2(n))=1.
\]
This result shows that from degree $4$, an algebra $\mathcal{A}s_3$ behaves like a free associative commutative algebra.

From Theorem \ref{basisAs3}, Lemma \ref{id1as3} and Lemma \ref{id2as3}, we immediately obtain
\begin{corollary}
A free nilpotent Lie algebra of index $4$ can be embedded into a free algebra from the variety $\mathcal{A}s_3$ under the commutator.
\end{corollary}
\begin{corollary}
A free Jordan algebra with identities from Lemma \ref{id1as3} can be embedded into a free algebra from the variety $\mathcal{A}s_3$ under the anti-commutator.
\end{corollary}

\begin{remark}
In \cite{Lopatin}, there was constructed a basis of algebra $\mathcal{A}s_1\<X\>$, i.e., a free associative algebra with identity
\[
abc+acb+bac+bca+cab+cba=0.
\]
It is worth noting that the algebra $\mathcal{A}s_1\<X\>$ is nilpotent of index $6$.
\end{remark}

\subsection*{Acknowledgments}
This research was funded by the Science Committee of the Ministry of Science and Higher Education of the Republic of Kazakhstan (Grant No. AP23484665).

\end{document}